\documentclass[12pt]{article}
\usepackage{amsmath,amssymb,amsthm}
\usepackage{graphicx}
\usepackage{subfigure}
\usepackage{psfrag}
\usepackage{color}
\usepackage{enumerate}
\usepackage{epstopdf}
\newtheorem{thm}{Theorem}[section]

\newtheorem{prop}[thm]{Proposition}
\theoremstyle{definition}
\newtheorem{prob}[thm]{Problem}

\def\less{\backslash}

\def\marker{\>\hbox{${\vcenter{\vbox{
    \hrule height 0.4pt\hbox{\vrule width 0.4pt height 6pt
    \kern6pt\vrule width 0.4pt}\hrule height 0.4pt}}}$}\>}

\newcommand{\es}{{\rm es}_{\chi}}

\title{Some extremal results on the chromatic-stability index}

\author{\small {Shenwei Huang$^1$, Sandi Klav\v{z}ar$^{2,3,4}$\footnote{The corresponding author.},\ Hui Lei$^5$, Xiaopan Lian$^6$ and Yongtang Shi$^6$}\\
{\small $^1$ College of Computer Science, Nankai University, Tianjin 300350, China}\\
{\small $^2$ Faculty of Mathematics and Physics, University of Ljubljana, Slovenia}\\
{\small $^3$ Faculty of Natural Sciences and Mathematics, University of Maribor, Slovenia}\\
{\small $^4$ Institute of Mathematics, Physics and Mechanics, Ljubljana, Slovenia}\\
{\small $^5$ School of Statistics and Data Science, LPMC and KLMDASR}\\
{\small Nankai University, Tianjin 300071, China}\\
{\small $^6$ Center for Combinatorics and LPMC, Nankai University, Tianjin, China}\\
{\small Email: shenweihuang@nankai.edu.cn; sandi.klavzar@fmf.uni-lj.si}\\ {\small hlei@nankai.edu.cn; xiaopanlian@mail.nankai.edu.cn; shi@nankai.edu.cn}\\
}
\date{\today}

\begin{document}
\maketitle
\begin{abstract}
The $\chi$-stability index ${\rm es}_{\chi}(G)$ of a graph $G$ is the minimum number of its edges whose removal results in a graph with the chromatic number smaller than that of $G$. In this paper three open problems from [European J.\ Combin.\ 84 (2020) 103042] are considered. Examples are constructed which demonstrate that a known characterization of $k$-regular ($k\le 5$) graphs $G$ with ${\rm es}_{\chi}(G) = 1$ does not extend to $k\ge 6$. Graphs $G$ with $\chi(G)=3$ for which ${\rm es}_{\chi}(G)+{\rm es}_{\chi}(\overline{G}) = 2$ holds are characterized. Necessary conditions on graphs $G$ which attain a known upper bound on ${\rm es}_{\chi}(G)$ in terms of the order and the chromatic number of $G$ are derived. The conditions are proved to be sufficient when $n\equiv 2  \pmod 3$ and $\chi(G)=3$.
\end{abstract}


\noindent
{\bf AMS Subj.\ Class.\ (2020)}: 05C15, 05C35

\section{Introduction}

\baselineskip 17pt
If ${\mathcal I}$ is a graph invariant and $G$ a graph, then it is natural to consider the minimum number of vertices of $G$ whose removal results in an induced subgraph $G'$ with ${\mathcal I}(G') \ne {\mathcal I}(G)$ or with $E(G') = \emptyset$, see~\cite{alikhani-2020+}. Let us call this number the {\em ${\mathcal I}$-stability number} of $G$ and denote it by $vs_{\mathcal I}(G)$. Similarly one can be interested in the minimum number of edges that has to be removed in order to obtain a spanning subgraph $G'$ with ${\mathcal I}(G') \ne {\mathcal I}(G)$ or with $E(G') = \emptyset$. In this case let us call the minimum number of edges the {\em ${\mathcal I}$-stability index} of $G$ and denote it by $es_{\mathcal I}(G)$.

In this paper we are interested in the $\chi$-stability index $\es$, spelled out as {\em chromatic-stability index}. The $\chi$-stability index $\es(G)$ of a graph $G$ with at least one edge is thus the minimum number of edges of $G$ whose removal results in a graph with the chromatic number smaller than that of $G$. If $E(G) = \emptyset$, then $\es(G) = 0$. It should be noted that in some papers the term ``chromatic edge-stability number" was used, but within the above proposed general framework, as well as since the investigation of the $\chi'$-stability number has been initiated in~\cite{alikhani-2020+}, this earlier naming would lead to a confusing terminology.

The $\chi$-stability index was first studied by Staton~\cite{S1980}, who provided upper bounds $\es$ for regular graphs in terms of the size of a given graph. The invariant was subsequently  investigated in~\cite{arumugam-2008, BKM2020, KMM2018}. In this paper we continue this line of the research and are primarily interested in the following three open problems on the chromatic-stability index.

\begin{prob} [\cite{AKMN2020, arumugam-2008}]
\label{Oprob5.3}
Characterize graphs $G$ with $\es(G)=1$.
\end{prob}

\begin{prob} [\cite{AKMN2020}]
\label{Oprob5.2}
Characterize graphs $G$ with $\es(G)+\es(\overline{G})=2$.
\end{prob}

In~\cite{AKMN2020} it was proved that if $G$ is a graph of order $n$ with $r=\chi(G)$, then
\begin{equation}
\label{eq:upper}
\es(G)\leq\begin{cases}
		\lfloor\frac{n}{r}\rfloor\lfloor\frac{n}{r}+1\rfloor\,;      & n\equiv r-1 \pmod r,  \\[2mm]
		\lfloor\frac{n}{r}\rfloor^2\ ;  & \text{otherwise}.\\
	\end{cases}
\end{equation}

The third open problem of our interest now read as follows.

\begin{prob} [\cite{AKMN2020}]
\label{Oprob5.1}
Characterize graphs that attain the upper bound in~\eqref{eq:upper}.
\end{prob}

In the rest of this section we recall definitions needed in this paper. In Section~\ref{sec:first-open} we consider graphs $G$ with $\es(G) = 1$ and construct examples which demonstrate that a known characterization of $k$-regular graphs $G$ with $\es(G) = 1$ does not extend to $k\ge 6$. Then, in Section~\ref{sec:second-open}, we characterize graphs $G$ with $\chi(G)=3$ for which $\es(G)+\es(\overline{G}) = 2$ holds. In the concluding section we obtain necessary structural conditions on graphs $G$ which  attain the upper bound in~\eqref{eq:upper}. The conditions are proved to be sufficient when $n\equiv 2  \pmod 3$ and $\chi(G)=3$.

The {\it chromatic number} $\chi(G)$ of a graph $G$ is the smallest integer $k$ such that $G$ admits a proper coloring of its vertices using $k$ colors. Unless stated otherwise, we will assume that the colors are from the set $[k] = \{1,\ldots, k\}$.  A {\it $\chi(G)$-coloring}, or simply {\it $\chi$-coloring} of $G$ is a proper coloring using $\chi(G)$ colors. In a coloring of $G$, a set of vertices having the same color form a {\it color class}. If $c$ is a $k$-coloring of $G$ with color classes $C_1, \ldots, C_k$, then we will identify $c$ with $(C_1,\ldots, C_k)$, that is, we will say that $c$ is a coloring $(C_1,\ldots, C_k)$. When we will wish to emphasize that these color classes correspond to $c$, we will denote them by $(C_1^c,\ldots, C_k^c)$. If $c$ is a coloring of  $G$ and $A\subseteq V(G)$, then let $c(A) = \bigcup_{a\in A} c(a)$. Let $c^*(G)$ denote the cardinality of a smallest color class among all $\chi$-colorings of $G$. If $c^*(G) = 1$, then we say that $G$ has a {\it singleton color class}. The {\it chromatic bondage number} $\rho(G)$ of $G$ denotes the minimum number of edges between two color classes among all $\chi$-colorings of a graph $G$. Note that  $\es(G)\leq \rho(G)$ clearly holds.

For $v\in V(G)$, let $d_G(v)$ and $N_G(v)$ denote the degree and the open neighborhood of $v$ in $G$, respectively. If $A\subseteq V(G)$, then let $N_G(A)=(\cup_{v\in A}N_G(v))\less A$. For $A,B\subseteq V(G)$, let $E[A,B]$ be the set of edges which have one endpoint in $A$ and the other in $B$, and let $e(A,B)=|E[A,B]|$. The subgraph of $G$ induced by $A\subseteq V(G)$ will be denoted by $G[A]$. The  {\it girth} $g(G)$ of a graph $G$ is the length of a shortest cycle in $G$. The order of a largest complete subgraph in $G$ is the {\em clique number} $\omega(G)$ of $G$. The {\it complement} of $G$ is denoted by $\overline{G}$.

\section{On Problem \ref{Oprob5.3}}
\label{sec:first-open}

Problem~\ref{Oprob5.3} which asks for a characterization of graphs $G$ with $\es(G)=1$ has been independently posed in~\cite[Problem 2.18]{arumugam-2008} and in~\cite[Problem 5.3]{AKMN2020}. The two equivalent reformulations of the condition $\es(G)=1$ from the next proposition are due to~\cite[Proposition 2.2]{KMM2018} and~\cite[Remark 2.15]{arumugam-2008}, respectively. To be self-contained, we include a simple proof of the result.

\begin{prop}\label{prop:easy-char}
If $G$ is a graph with $\chi(G) \ge 2$, then the following claims are equivalent.
\begin{description}
\item[(i)] $\es(G)=1$.
\item[(ii)] $\rho(G) = 1$.
\item[(iii)] $G$ admits a $\chi(G)$-coloring $(C_1, \ldots, C_{\chi(G)})$, where $|C_1| = 1$ and $e(C_1,C_2) = 1$.
\end{description}
\end{prop}

\begin{proof}
Let $\es(G)=1$ and let $e=uv\in E(G)$ be an edge such that $\chi(G-e) = \chi(G) - 1$. If $c$ is a $(\chi(G) - 1)$-coloring of $G-e$, then $c(u) = c(v)$, for otherwise $c$ would be a proper coloring of $G$ (using only $\chi(G) - 1$ colors). Recoloring $u$ with a new color yields a coloring of $G$ as required by (iii). Hence (i) implies (iii). The implication (iii) $\Rightarrow$ (ii) is obvious, and (ii) $\Rightarrow$ (i) follows from the already noted fact that $\es(G) \le \rho(G)$ holds.
\end{proof}

Although Proposition~\ref{prop:easy-char} formally gives two characterizations of graphs $G$ with $\es(G)=1$, it should be understood that Problem~\ref{Oprob5.3} asks for a {\em structural characterization} of such graphs. A partial solution of the problem is provided in the following result.

\begin{thm}  {\rm (\cite[Theorem 4.4]{AKMN2020})}.
\label{thm:k-regular}
Let $G$ be a connected, $k$-regular graph, $k\leq 5$. Then $\es(G)=1$ if and only if $G$ is $K_2$, $G$ is an odd cycle, or $\chi(G)>3$ and $c^{\star}(G) = 1$.
\end{thm}

The second part of~\cite[Problem 5.3]{AKMN2020} says: ``In particular, for the regular case extend the classification of Theorem~\ref{thm:k-regular} to $k>5$." We do not solve the problem, but demonstrate in the rest of the section that (i) the problem appears difficult and (ii) why $k=5$ is the threshold for regular graphs. Let $X$ be the graph as drawn in Fig.~\ref{fig1}.

\begin{figure}[ht!]
\begin{center}
\scalebox{0.3}[0.3]{\includegraphics{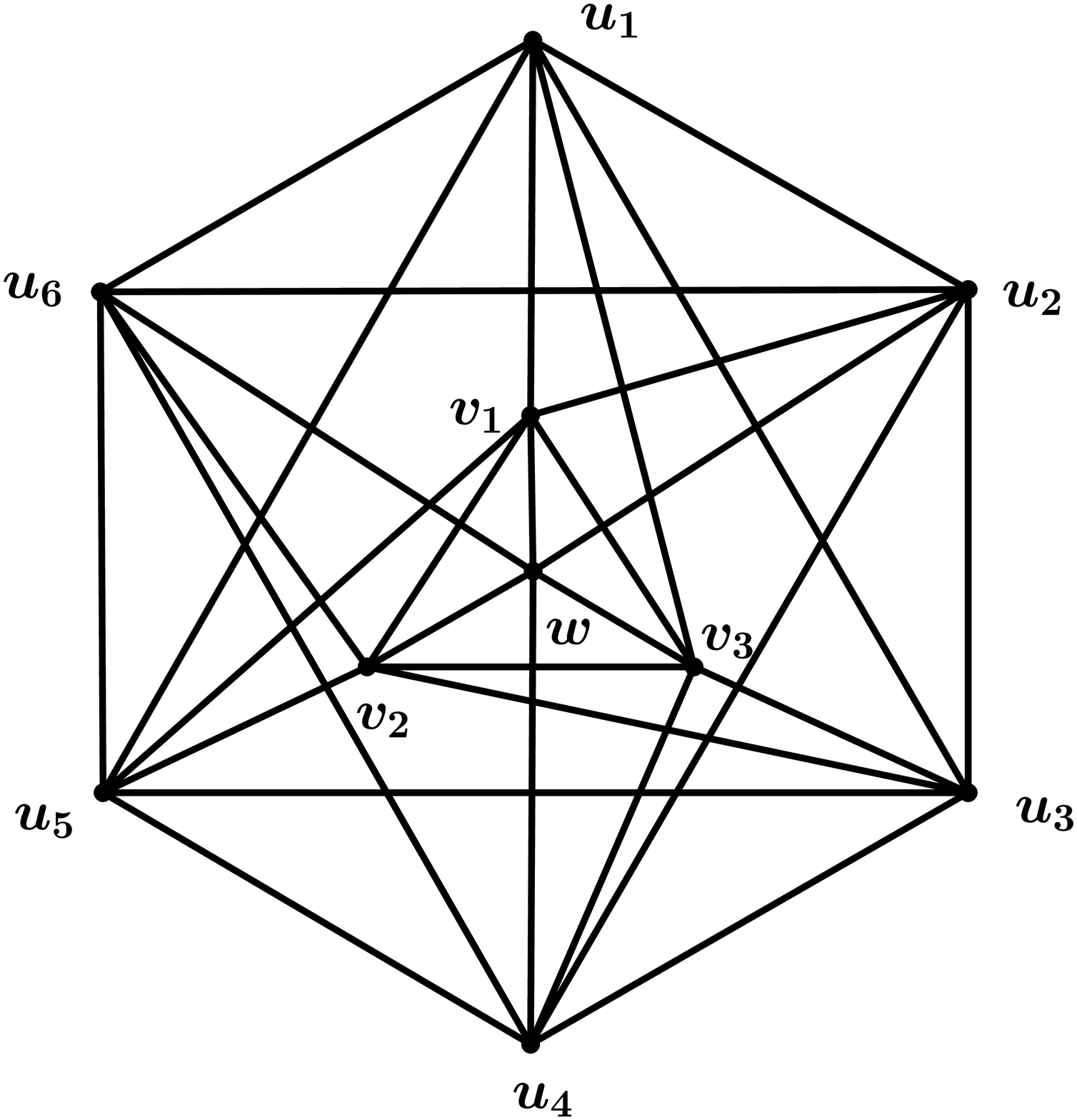}}
\caption{Graph $X$.}\label{fig1}
\end{center}
\end{figure}

Then we have:

\begin{prop}\label{prop1}
The graph $X$ is a $6$-regular graph with $\chi(X)=4$, $c^{\star}(X) = 1$, and $\es(X) = 2$.
\end{prop}

\begin{proof}
Since $\omega(X)=4$,  $\chi(X)\geq 4$. We give a 4-coloring $c$ of $X$ as follows: $c(w)=4$, $c(v_1)=c(u_3)=c(u_6)=1$, $c(v_3)=c(u_2)=c(u_5)=2$, $c(v_2)=c(u_1)=c(u_4)=3$. Since color 4 is used exactly once, $\chi(X)=4$ and $c^*(X)=1$. It remains to prove that $\es(X) = 2$.

Let $X'$ be the graph obtained from $X$ by deleting the edges $wv_1, wu_6$. Then we can get a  3-coloring $c'$ of $X'$ as follows: $c'(w) = c'(v_1) = c'(u_3) = c'(u_6) = 1$, $c'(v_3) = c'(u_2) = c'(u_5) = 2$, and $c'(v_2) = c'(u_1) = c'(u_4) = 3$. Hence $\es(X) \le 2$.

Suppose now on the contrary that $\es(X) = 1$. Then by~Proposition~\ref{prop:easy-char}(iii),  there exists a coloring  $c = (C_1, C_2, C_3, C_4)$, such that $|C_1| = 1$ and  $e(C_1, C_2)=1$. Since $X[\{v_1,v_2,v_3,w\}]\cong K_4$, we have $c(w)=1$ or $c(v_i)=1$ for some $i\in [3]$. If $c(w)=1$, then $\chi(X[N(w)])=3 $ and color 2 appears only once in $N(w)$. But this is impossible because $X[v_1,v_2,v_3]\cong K_3$ and $X[u_2,u_4,u_6]\cong K_3$. If $c(w)\ne 1$, then by symmetry we may without loss of generality assume that $c(v_1)=1$. Then we consider the coloring of $N(v_1)$. If $c(u_5)\neq c(v_3)$, say $c(u_5)=a\in\{2,3,4\}$ and $c(v_3)=b\in \{2,3,4\}\less\{a\}$, then $c(v_2)=c(u_1)=c=\{2,3,4\}\less\{a,b\}$, $c(w)=a$, and $c(u_2)=b$, contradicting  the fact that $e(C_1, C_2)=1$. If $c(u_5)=c(v_3)$, say $c(u_5)=c(v_3)=a\in\{2,3,4\}$, then $c(v_2)=b\in \{2,3,4\}\less\{a\}$, and $c(w)=c=\{2,3,4\}\less\{a,b\}$. Since $\{w,v_2,u_5\}\subseteq N(u_6)$, we have $c(u_6)=1$, a contradiction with the fact that $|C_1| = 1$. So $\es(G)\geq 2$ and we are done.
\end{proof}

Proposition~\ref{prop1} shows that Theorem~\ref{thm:k-regular} does not extend to $6$-regular graphs. On the other hand, consider the following example to see that there exist $4$-chromatic, $6$-regular (and of higher regularity) graphs with $\es(G)=1$.  A graph $G=C(n; a_0,a_1, \ldots, a_k)$ is called a {\it circulant} if $V(G)=[n]$ and $E(G)=\{(i,j): |i-j|\in\{a_0, a_1, \ldots, a_k\} \pmod n\}$, where $1\leq a_0<a_1<\cdots <a_k\leq n/2$.
If $a_k<n/2$, then $G$ is a $(2k+2)$-regular graph; otherwise, $G$ is $(2k + 1)$-regular. In \cite[Theorem 2.1]{DMP2004}, Dobrynin, Melnikov, and Pyatkin constructed 4-critical $r$-regular circulants for $r\in\{6,8,10\}$. (Recall that a graph $G$ with $\chi(G)=k$ is called {\it edge-critical} (or simply {\it $k$-critical}) if its chromatic number is strictly less than $k$ after removing any edge.) Hence these regular graphs satisfy $\es(G)=1$.

\section{On Problem \ref{Oprob5.2}}
\label{sec:second-open}

Let $G$ be a graph with $\es(G)=1$ and $\chi(G)=r$. We say that a $\chi$-coloring of $G$ is {\it a good coloring} if it satisfies the conditions of Proposition~\ref{prop:easy-char}(iii). Let $\mathcal{C}(G)$ be the set of good colorings of $G$. If $c = (C^c_1, \ldots, C^c_r) \in \mathcal{C}(G)$, then we may always without loss of generality assume that $|C^c_1|=1$ and $e(C^c_1, C^c_2)=1$.

Clearly, $\es(G)+\es(\overline{G}) = 2$ holds if and only if $\es(G) = \es(\overline{G}) = 1$. We first characterize disconnected graphs $G$ for which $\es(G)+\es(\overline{G}) = 2$ holds.

\begin{prop}
Let $G$ be a graph with components $G_1, \ldots,G_s$, $s\ge 2$, and let $\mathcal{G}=\{G_i:\  \chi(G_i)=\chi(G), i\in[s]\}$. Then $\es(G)+\es(\overline{G})=2$ if and only if
\begin{description}
\item[(i)] $|\mathcal{G}|=1$ and $\es(G_i)=1$ for $G_i\in \mathcal{G}$, and
\item[(ii)] there exists a $G_j$ such that $\es(\overline{G_j})=1$, or there exist components $G_j$ and $G_k$, $j\neq k$, such that $c^*(\overline{G_j})=1$ and $c^*(\overline{G_k})=1$.
\end{description}
\end{prop}

\begin{proof}
The following fact is essential for the rest of the argument: if $c$ is a proper coloring of $\overline{G}$, then $c(V(G_i))\cap c(V(G_j)) = \emptyset$ for every $i, j\in [s]$, $i\ne j$. If $G$ satisfies (i) and (ii), then (i) yields $\es(G)=1$, while (ii) gives $\es(\overline{G})=1$. Conversely, suppose that $\es(G)+\es(\overline{G})=2$. Then $\es(G)=1$ and $\es(\overline{G})=1$. If  $|\mathcal{G}|\geq2$ or $\es(G_i)\geq 2$ for any $G_i\in \mathcal{G}$, then $\chi(G - e)=\chi(G)$ for any $e\in E(G)$, a contradiction. This means that (i) holds.  Since $\es(\overline{G})=1$, there exists an a edge $\overline{e}\in E(\overline{G})$ such that $\chi(\overline{G} - \overline{e})<\chi(\overline{G})$. We consider two cases for the edge $\overline{e}$. If $\overline{e}\in E(\overline{G_j})$ for some $j\in [s]$, then $\es(\overline{G_j})=1$. In the other case the two endpoints of $\overline{e}$ lie in different components, say in $G_j$ and in  $G_k$, $j\ne k$. But then $c^*(\overline{G_j})=1$ and $c^*(\overline{G_k})=1$. Thus (ii) holds as well.
\end{proof}

In the main result of this section we now characterize connected graphs $G$ with $\chi(G)=3$ for which $\es(G)+\es(\overline{G})=2$ holds.

\begin{thm}\label{NG3}
Let $G$ be a connected graph of order $n$, with $\chi(G)=3$. Then $\es(G)+\es(\overline{G})=2$  if and only if
\begin{description}
  \item[(i)] all odd cycles in $G$ share one edge,
  \item[(ii)] $c^*(\overline{G})=1$,
  \item[(iii)] $\chi(\overline{G})\geq\lceil\frac{n}{2}\rceil$,
  \item[(iv)] if $n$ is even, $\chi(\overline{G})=\frac{n}{2}$, and $||C^c_2|-|C^c_3||=1$ for each $c = (C_1^c, C_2^c, C_3^c)\in \mathcal{C}(G)$, then $g(G)=3$, and for any proper coloring of $\overline{G}$, if $\{x_1,x_2,x_3\}$ is a color class, then $d_{G}(v)\geq2$ for each  $v\in N_G(\{x_1,x_2,x_3\})$.
\end{description}
\end{thm}

\begin{proof}
{\it Necessity:} Since $\es(G)=1$ and $\chi(G)=3$, there is an edge $e\in E(G)$ such that $G - e$ has no odd cycles. So (i) holds. It was observed in~\cite[Lemma 4.3]{AKMN2020} that  $\es(G)=1$ implies $c^*(G)=1$, hence (ii) holds. Let $c = (C^c_1,C^c_2, C^c_3)\in \mathcal{C}(G)$. We have $\omega(\overline{G})\geq\lfloor\frac{n}{2}\rfloor$ since $|C^c_1|=1$. So, $\chi(\overline{G})\geq \omega(\overline{G})\geq\lceil\frac{n}{2}\rceil$ when $n$ is even. In the case of $n$ is odd and $\omega(\overline{G})=\frac{n-1}{2}$, we have $|C^c_2|=|C^c_3|=\frac{n-1}{2}$ and $\overline{G}[C^c_2]\cong K_{C^c_2}$, $\overline{G}[C^c_3]\cong K_{C^c_3}$. Note that for any proper coloring of $\overline{G}$, there is at most one color class with 3 vertices, and the number of vertices in other color classes must be smaller than $3$. By Proposition~\ref{prop:easy-char}(iii),  there exists a $\chi$-coloring of $\overline{G}$ such that some color class has exactly one vertex. Then $\chi(\overline{G})\geq \frac{n+1}{2}=\lceil\frac{n}{2}\rceil$.

Suppose now that $n$ is even, $\chi(\overline{G})=\frac{n}{2}$, and $||C^c_2|-|C^c_3||=1$ for any $c = (C_1^c, C_2^c, C_3^c)\in \mathcal{C}(G)$. Let $C^c_1=\{x_1\}$, and let $x_2$ be the vertex of $C^c_2$ such that  $x_1x_2\in E(G)$.
Let $\bar{c}\in \mathcal{C}(\overline{G})$ and let the color set used by $\bar{c}$ be $[\frac{n}{2}]$. We claim that $\bar{c}(x_1)=\bar{c}(x_2)$ and $\bar{c}(x_1)\in \bar{c}(C^c_3)$. Notice that $x_1$  is in $\overline{G}$ adjacent to all vertices of $C^c_2$ except $x_2$. If $|C^c_2|-|C^c_3|=1$, then $|C^c_2|=\frac{n}{2}$. Then the claim holds because $\chi(\overline{G})=\frac{n}{2} = |\bar{c}(C^c_2)|$. Suppose second that $|C^c_3|-|C^c_2|=1$. Then $|C^c_3|=\frac{n}{2}$ and $|C^c_2|=\frac{n-2}{2}$. We have $|\bar{c}(C^c_3)| = \frac{n}{2}$. If $\bar{c}(x_1)\neq\bar{c}(x_2)$, then $\bar{c}(x_1\cup C^c_2)=[\frac{n}{2}]$,  contradicting the fact that $\bar{c}\in \mathcal{C}(\overline{G})$ because there is no singleton color class. Hence $\bar{c}(x_1)=\bar{c}(x_2)$ and $\bar{c}(x_1)\in \bar{c}(C^c_3)$ since $\bar{c}(C^c_3)=[\frac{n}{2}]$. Thus
$g(G)=3$. We might as well set $x_3\in C^c_3$ and $\bar{c}(x_1)=\bar{c}(x_2)=\bar{c}(x_3)=1$ in the following.
Suppose there is a vertex $v\in N_G(\{x_1,x_2,x_3\})$ such that $d_G(v)=1$.  If $|C^c_2|-|C^c_3|=1$,  then we have $C^c_2\subseteq N_{\overline{G}}(v)$ when $v\in N_G(x_1)$, $(C^c_2\less \{x_2\})\cup x_3\subseteq N_{\overline{G}}(v)$ when $v\in N_G(x_2)$ and $V(G)\less \{x_3\}= N_{\overline{G}}(v)$ when $v\in N_G(x_3)$.
Thus $\chi(\overline{G})>\frac{n}{2}$ when $v\in N_G(x_1)\cup N_G(x_2)$, a contradiction. When $v\in N_G(x_3)$,  we may without loss of generality assume that $\bar{c}(C^c_3)=[\frac{n-2}{2}]$.  Then $\bar{c}(v)=\frac{n}{2}$ since $x_2\in N_{\overline{G}}(v)$. But every color in $[\frac{n-2}{2}]$ appears exactly twice in $N_{\overline{G}}(v)$,  contradicting the fact that $\bar{c}\in \mathcal{C}(\overline{G})$.
If $|C^c_3|-|C^c_2|=1$, then we have $\chi(\overline{G})>\frac{n}{2}$ when $v\in N_G(x_3)$ and  $\bar{c}\notin \mathcal{C}(\overline{G})$  when $v\in N_G(x_1)\cup N_G(x_2)$ by the same analysis above, a contradiction.
\medskip

{\it Sufficiency:} Suppose an edge $e$ is shared by all odd cycles of $G$.  Then $\chi(G -e)\leq 2$. Hence $\es(G)=1$ holds by definition.
Suppose $\chi(\overline{G})\geq\lceil\frac{n}{2}\rceil$.  In~\cite[Lemma 4.2]{AKMN2020} it was proved that if $\chi(\overline{G})\geq\frac{n+2}{2}$, then $\es(\overline{G})=1$. So we may assume $\chi(\overline{G})=\lceil\frac{n}{2}\rceil$ in the following.

Suppose first that $n$ is odd. Let $\bar{c}$ be a proper coloring of $\overline{G}$. Since $\chi(\overline{G})=\lceil\frac{n}{2}\rceil=\frac{n+1}{2}$, the complement $\overline{G}$ has a singleton color class under $\bar{c}$. If $\overline{G}$ has two singleton color classes under $\bar{c}$, then $\rho(\overline{G})=1$. Otherwise, other color classes have exactly two vertices. At this time,  since $G$ is connected, $\Delta(\overline{G})<n-1$, thus $\rho(\overline{G})=1$.
Suppose second that $n$ is even. We have $||C^c_2|-|C^c_3||=1$ for any $c_i\in \mathcal{C}(G)$ since $\chi(\overline{G})=\frac{n}{2}$. Since
$c^*(\overline{G})=1$, there is a proper coloring such that some color class contains three vertices. Let $\bar{c}$ be  the proper coloring and $\bar{c}(x_1)=\bar{c}(x_2)=\bar{c}(x_3)$, where $x_s\in C^c_s $ for $s\in[3]$.
Let $\{\alpha,\beta\}=\{2,3\}$.
If $|C^c_\alpha|-|C^c_\beta|=1$, then $|C^c_\alpha|=\frac{n}{2}$ and $|C^c_\beta|=\frac{n-2}{2}$. Since $\chi(\overline{G})=\frac{n}{2}$, we may assume $\bar{c}(C^c_\alpha)=[\frac{n}{2}]$ and $\frac{n}{2}\notin C^c_\beta$, say $\bar{c}(u)=\frac{n}{2}$. Since $G$ is connected and $d_G(v)\geq2$ for any $v\in N_G(\{x_1,x_2,x_3\})$, we have $N_{C^c_\beta}(u)\less \{x_\beta\}\neq \emptyset$ or $\{x_1,x_\beta\}\subseteq N_G(u)$.
Thus $\rho(\overline{G})=1$, and by Proposition~\ref{prop:easy-char} we conclude that $\es(\overline{G})=1$.
\end{proof}

\section{On Problem \ref{Oprob5.1}}
\label{sec:third-open}

Obviously, when $r=2$, the upper bound in~\eqref{eq:upper} is attained if and only if the graph in question is a complete  bipartite graph in which the orders of its bipartition sets differ by at most one. For an arbitrary $r$ we have:

\begin{thm}\label{thm2.1}
Let $G$ be a graph of order $n$ and with $r=\chi(G)$.
\begin{description}
  \item[(i)] Suppose that $n\equiv r-1 \pmod r$ and  $\es(G)=\lfloor \frac{n}{r}\rfloor\lfloor \frac{n}{r}+1\rfloor$. Then for any $r$-coloring $(C_1,\ldots, C_r)$ of $G$, where $|C_1|\leq  \cdots \leq |C_r|$,  we have
  \begin{description}
  \item[(1)] $|C_1|=\lfloor\frac{n}{r}\rfloor$, and  $|C_2|=\cdots=|C_r|=\lfloor\frac{n}{r}+1\rfloor$.
  \item[(2)] If $2\leq i\leq r$, then  $G[C_1\cup C_i]$ is a complete bipartite graph with bipartition $(C_1, C_i)$.
  \item[(3)] If $v\in C_i$ and $j\in[r]\less\{i\}$, then $e(v, C_j)\geq\lfloor\frac{n}{r}\rfloor$.
\end{description}

\item[(ii)] Suppose that $n\not\equiv r-1 \pmod r$ and $\es(G)=\lfloor \frac{n}{r}\rfloor^2$. Then for any $r$-coloring $(C_1,\ldots, C_r)$ of $G$, where $|C_1|\leq  \cdots \leq |C_r|$,  we have
\begin{description}
  \item[(1)] $|C_1|=|C_2|=\lfloor\frac{n}{r}\rfloor$.
  \item[(2)] If $|C_i|=\lfloor\frac{n}{r}\rfloor$, and $v\in C_i$ and $j\in[r]\less\{i\}$, then $e(v, C_j)\geq\lfloor\frac{n}{r}\rfloor$. If $|C_i|>\lfloor\frac{n}{r}\rfloor$, then $\sum_{v_s\in C_i} \ell_s\geq {\lfloor\frac{n}{r}\rfloor}^2$, where $\ell_s=\min\{e(v_s, C_j):\ v_s\in C_i$, $j\in[r]\less\{i\}\}$.
\end{description}
\end{description}
\end{thm}

\begin{proof}
\indent (i) Consider an $r$-coloring $(C_1,\ldots, C_r)$ of $G$, where $|C_1|\leq  \cdots \leq |C_r|$.  \\[0.8ex]
\indent (1) Since $n\equiv r-1 \pmod r$, we have $n = r\lfloor n/r\rfloor + r -1$. From here it was deduced in the proof of~\cite[Theorem 2.1]{AKMN2020} that there exists at least one pair of color class $C_i$ and $C_j$, $i<j$, such that $|C_i|+|C_j|\leq\lfloor\frac{n}{r}\rfloor+\lfloor\frac{n}{r}+1\rfloor$. Since $\es(G)=\lfloor \frac{n}{r}\rfloor\lfloor \frac{n}{r}+1\rfloor$, we have $|C_i|=\lfloor\frac{n}{r}\rfloor$ and $|C_j|=\lfloor\frac{n}{r}+1\rfloor$. Moveover, we have $i=1$ and $|C_k|\geq\lfloor\frac{n}{r}+1\rfloor$ for $2\leq k\leq r$, since otherwise $\es(G)\leq |C_1||C_2|\leq \lfloor \frac{n}{r}\rfloor^2$, a contradiction. Thus
$|C_2|=\cdots=|C_r|=\lfloor\frac{n}{r}+1\rfloor$ because $n = r\lfloor n/r\rfloor + r -1$. \\[0.8ex]
\indent (2, 3) Observe that $G[C_1\cup C_i]$ is a complete bipartite graph with bipartition $(C_1, C_i)$ for any $2\leq i\leq r$, since otherwise,
$\es(G)\leq |C_1||C_i|\leq \lfloor \frac{n}{r}\rfloor\lfloor \frac{n}{r}+1\rfloor-1$, a contradiction.
Therefore, we have $e(v, C_j)\geq\lfloor\frac{n}{r}\rfloor$ when $v\in C_1$ or $j=1$. If $e(v, C_j)<\lfloor\frac{n}{r}\rfloor$ for some $v\in C_i$  and $j\in[r]\less\{i\}$ ($i,j>1$), then by deleting the edge set $E(v, C_j)\cup E(C_i\less\{v\}, C_1)$, we get an $(r-1)$-coloring
with the color class set $\{C_1\cup(C_i\less\{v\}), C_2,\ldots,C_j\cup\{v\},\ldots, C_r\}\less \{C_i\}$. Notice that $|E(v, C_j)\cup E(C_i\less\{v\}, C_1)|<\lfloor \frac{n}{r}\rfloor\lfloor \frac{n}{r}+1\rfloor$. Thus $\es(G)< \lfloor \frac{n}{r}\rfloor\lfloor \frac{n}{r}+1\rfloor$, a contradiction. \\[0.8ex]
\indent (ii) Suppose $n\not\equiv r-1 \pmod r$  and $\es(G)=\lfloor \frac{n}{r}\rfloor^2$. Consider an $r$-coloring $(C_1,\ldots, C_r)$ of $G$, where $|C_1|\leq |C_2|\leq \cdots \leq |C_r|$.\\[0.8ex]
\indent (1) By the proof of~\cite[Theorem 2.1]{AKMN2020},  there exists at least one pair of color class $C_i$ and $C_j$ ($i\leq j$) in which
$|C_i|+|C_j|\leq2\lfloor\frac{n}{r}\rfloor$. Since $\es(G)=\lfloor \frac{n}{r}\rfloor^2$, we have $|C_i|=|C_j|=\lfloor\frac{n}{r}\rfloor$. Moveover, we have $|C_k|\geq\lfloor\frac{n}{r}\rfloor$ for $1\leq k\leq r$, since otherwise $\es(G)\leq |C_1||C_2|< \lfloor \frac{n}{r}\rfloor^2$, a contradiction. Thus $|C_1|=|C_2|=\lfloor\frac{n}{r}\rfloor$. 
\\[0.8ex]
\indent (2) Suppose $|C_i|=\lfloor\frac{n}{r}\rfloor$ and there exists some $v\in C_i$ and $j\in[r]\less\{i\}$ such that $e(v, C_j)<\lfloor\frac{n}{r}\rfloor$. We take a color class $C_k$ with $\lfloor\frac{n}{r}\rfloor$ vertices, which is different from $C_i$. This is possible because $|C_1|=|C_2|=\lfloor\frac{n}{r}\rfloor$.  Note that $k$ and $j$ are not necessarily distinct. Then we delete the edge set $E(v, C_j)\cup E(C_i\less\{v\}, C_k)$ and get an $(r-1)$-coloring
with color class set $\{C_1, \ldots, C_k\cup(C_i\less\{v\}), \ldots,C_j\cup\{v\},\ldots, C_r\}\less \{C_i\}$. Notice that $|E(v, C_j)\cup E(C_i\less\{v\}, C_k)|<\lfloor \frac{n}{r}\rfloor^2$. Thus $\es(G)< \lfloor \frac{n}{r}\rfloor^2$, a contradiction.

 Suppose $|C_i|>\lfloor\frac{n}{r}\rfloor$ and $\sum_{v_s\in C_i}\ell_s< {\lfloor\frac{n}{r}\rfloor}^2$, where $\ell_s=\min\{e(v_s, C_j):\ v_s\in C_i$, $j\in[r]\less\{i\}\}$. Let $C^s$ be one of the corresponding color classes when $\ell_s$ is taken for $v_s$. Then for any $v_s\in C_i$, we delete the edge set $E(v_s, C^s)$ and get an $(r-1)$-coloring by putting $v_s$ in $C^s$. Thus $\es(G)< \lfloor \frac{n}{r}\rfloor^2$, a contradiction.
\end{proof}

Recall that a graph coloring $(C_1,\ldots C_k)$ is {\em equitable}~\cite{erdos-1964} if $| |C_i| - |C_j| | \le 1$ holds for all $i\ne j$. Hence all the colorings from Theorem~\ref{thm2.1}(i) are equitable and consequently,  the corresponding extremal graphs have the same chromatic number and the equitable chromatic number. (See~\cite{heckel-2020, li-2020} for a couple of recent investigations of the equitable chromatic number.)

\begin{thm}
Let $G$ be a graph of order $n$, where $n\equiv 2 \pmod 3$, and with $\chi(G)=3$. If any  $3$-coloring of $G$ satisfies (1)-(3) of Theorem \ref{thm2.1}(i), then $\es(G)=\lfloor \frac{n}{3}\rfloor\lfloor \frac{n}{3}+1\rfloor$.
\end{thm}

\begin{proof}
Let $c$ be a $3$-coloring of $G$ satisfying (1)-(3) of Theorem \ref{thm2.1}(i). Let $\{i,j\}=\{2,3\}$. For $v\in C_i$ we may let $e(v, C_j)=\lfloor\frac{n}{3}\rfloor$ (as adding edges to a graph cannot decrease its $\chi$-stability index). Since for any $e\in E(G)$, $e$ lies in exactly $\lfloor\frac{n}{3}\rfloor$ subgraphs $K_3$, the graph $G - e$ has at most $\lfloor\frac{n}{3}\rfloor$  fewer subgraphs
isomorphic to $K_3$ than $G$. Let $F\subseteq E(G)$ with $|F|=\lfloor \frac{n}{3}\rfloor\lfloor \frac{n}{3}+1\rfloor-1$. Then the graph $G\less F$ has at most
$\lfloor \frac{n}{3}\rfloor(\lfloor \frac{n}{3}\rfloor\lfloor \frac{n}{3}+1\rfloor-1)$ fewer subgraphs $K_3$ than $G$. Since $G$ has $\lfloor \frac{n}{3}\rfloor\lfloor \frac{n}{3}\rfloor\lfloor \frac{n}{3}+1\rfloor$ subgraphs $K_3$, we thus infer
that $G\less F$ has at least one subgraph $K_3$ and consequently $\chi(G\less F)=3$. Hence, $\es(G)=\lfloor \frac{n}{3}\rfloor\lfloor \frac{n}{3}+1\rfloor$.
\end{proof}

Let $G$ be a graph with $n$ vertices and $r=\chi(G)$. Note that when $r=5$ and $n\equiv 4 \pmod 5$, the conditions (1)-(3) in Theorem~\ref{thm2.1}(i) are not sufficient. Let $G_{12}$ be the graph from Fig.~\ref{fig2}, and let $G_{14}$ be obtained from $G_{12}$ by adding two new vertices $u_0$ and $v_0$, and connecting $u_0$ and $v_0$ to all vertices of $G_{12}$. Then we have the following result.

\begin{figure}[htbp]
\begin{center}
\includegraphics[scale=0.3]{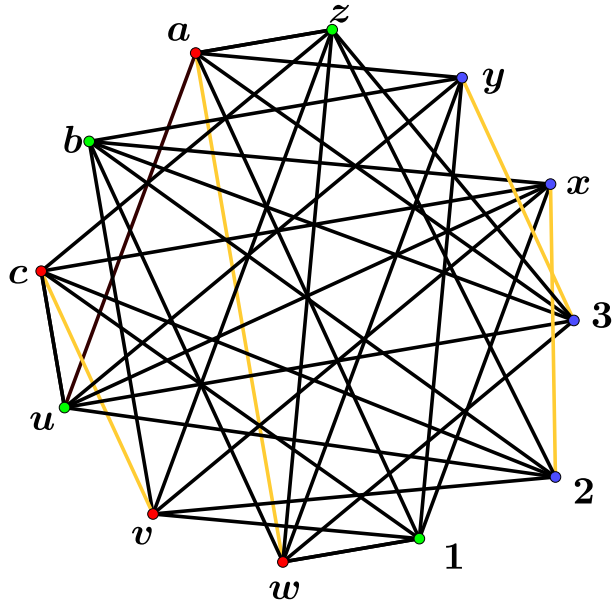}
\caption{Graph $G_{12}$.}\label{fig2}
\end{center}
\end{figure}

\begin{prop}\label{prop2}
The graph $G_{14}$ satisfies conditions (1)-(3) of Theorem~\ref{thm2.1}(i), but $es_\chi(G_{14})<\lfloor\frac{n}{r}\rfloor\lfloor\frac{n}{r}+1\rfloor=6$.
\end{prop}

\begin{proof}
We first show that $\chi(G_{14})=5$.
Let $A=\{a,b,c\}$, $B=\{u,v,w\}$, $C=\{1,2,3\}$, and $D=\{x,y,z\}$. We claim that $\chi(G_{14})=5$ and that $G_{14}$ has a unique $5$-coloring. With a computer search (using SageMath), we found all independent sets of $G_{14}$ with at least three vertices: $A$, $B$, $C$, $D$, $\{b,u,1,z\}$, and each $X\subseteq\{b,u,1,z\}$ with $|X|=3$. So, if any three vertices of $\{b,u,1,z\}$ have the same color under some proper coloring $c:V(G_{14})\rightarrow [k]$ of $G_{14}$, then $k\geq6$. Thus $\chi(G_{14})=5$ and the unique $5$-coloring has color classes $\{u_0,v_0\}$, $A$, $B$, $C$, $D$. Therefore, the graph $G_{14}$ satisfies conditions (1)-(3) of Theorem~\ref{thm2.1}(i).

On the other hand, by deleting the edges $cv$, $aw$, $3y$, and $2x$  (colored orange in the figure), we can get a $4$-coloring  with color classes $\{u_0,v_0\}$, $\{a,c,v,w\}$, $\{b,u,1,z\}$, $\{2,3,x,y\}$. Therefore, $\es(G_{14})\leq4$.
\end{proof}

\noindent{\bf Acknowledgements.}
Huang was partially supported by the National Natural Science Foundation of China (11801284) and the
Fundamental Research Funds for the Central Universities, Nankai University.
Lei, Lian and Shi were partially supported by the China-Slovenia
bilateral project ``Some topics in modern graph theory" (No. 12-6), the National Natural Science Foundation of China and the
Fundamental Research Funds for the Central Universities, Nankai University.
Klav\v zar acknowledges the financial support from the Slovenian Research Agency (research core funding No.\ P1-0297 and projects J1-9109, J1-1693, N1-0095, N1-0108).

\end{document}